\documentclass[reqno]{amsart}
\usepackage[english]{babel}
\usepackage{amssymb,verbatim}
\usepackage[T1]{fontenc}
\newtheorem{theorem}{Theorem}[section]

\newtheorem{lemma}[theorem]{Lemma}
\newtheorem{corollary}[theorem]{Corollary}
\theoremstyle{definition}
\newtheorem{definition}[theorem]{Definition}
\newtheorem{example}[theorem]{Example}
\theoremstyle{remark}
\newtheorem*{remark}{Remark}
\numberwithin{equation}{section}
\newcommand{\RE}{\mbox{$\mathbb{R}$}}
\newcommand{\C}{\mbox{$\mathbb{C}$}}
\begin{document}

\title{The geometry of $m$-hyperconvex domains}

\author{Per \AA hag}\address{Department of Mathematics and Mathematical Statistics\\ Ume\aa \ University\\SE-901~87 Ume\aa \\ Sweden}\email{Per.Ahag@math.umu.se}

\author{Rafa\l\ Czy{\.z}}\address{Faculty of Mathematics and Computer Science, Jagiellonian University, \L ojasiewicza~6, 30-348 Krak\'ow, Poland}
\thanks{The second-named author was partially supported by NCN grant DEC-2013/08/A/ST1/00312.}\email{Rafal.Czyz@im.uj.edu.pl}

\author{Lisa Hed}\address{Department of Mathematics and Mathematical Statistics\\ Ume\aa \ University\\SE-901~87 Ume\aa \\ Sweden}\email{Lisa.Hed@math.umu.se}


\begin{abstract} We study the geometry of $m$-regular domains within the Caffarelli-Nirenberg-Spruck model in terms of barrier functions, envelopes, exhaustion functions, and Jensen measures. We prove among other things that every $m$-hyperconvex domain admits an exhaustion function that is negative, smooth, strictly $m$-subharmonic, and has bounded $m$-Hessian measure.
\end{abstract}

\maketitle

\section{Introduction}

 The geometry of the underlying space is usually essential when studying a given problem in analysis. The starting
 point of this paper is the model presented by Caffarelli et al.~\cite{CHS} in 1985 that makes it possible to investigate the transition between potential and pluripotential theories. Their construction relies on G\aa rding's
 research on hyperbolic polynomials~\cite{Gaarding}. The authors of~\cite{CHS} also provided a very nice application to special Lagrangian geometry,
 which was in itself introduced as an example within calibrated geometry~\cite{HarveyLawson1}. With the publications
 of~\cite{Blocki_weak}, and~\cite{Li}, many analysts and geometers got their attention to the Caffarelli-Nirenberg-Spruck model. To
 mention some references~\cite{DinewDinew, huang, L2, N, phong, WanWang2, zhou}. A usual assumption in these studies is that the underlying domain should admit a continuous exhaustion function that is $m$-subharmonic in the sense of Caffarelli et al. (see Section~\ref{sec_prelim} for the definition of $m$-subharmonic functions). In this paper we shall study the geometric properties of these domains. Let us now give a thorough background on the motivation behind this paper. It all starts with the following theorem:

 \bigskip

 \noindent \textbf{Theorem~A.} Assume that $\Omega$ is a bounded domain in $\RE^N$, $N\geq 2$. Then the following
 assertions are equivalent.

 \medskip

 \begin{enumerate}\itemsep2mm

 \item $\partial\Omega$ is \emph{regular} at every boundary point $y_0\in\partial\Omega$, in the sense that
 \[
\lim_{x\to y_0\atop x\in\Omega} \operatorname{PWB}_f(x) =f(y_0)\, ,
 \]
for each continuous function $f:\partial\Omega\to \RE$. Here
\[
\operatorname{PWB}_f(x)=\sup\Bigg\{v(x): v\in\mathcal{SH}(\Omega),\; \varlimsup_{\zeta\rightarrow\xi \atop
\zeta\in\Omega}v(\zeta)\leq f(\xi)\, , \;\; \forall \xi\in\partial\Omega\Bigg\}\, ,
\]
and $\mathcal{SH}(\Omega)$ is the space of subharmonic functions defined on $\Omega$;

 \item $\partial\Omega$ has a \emph{strong barrier} at every point $y_0\in\partial\Omega$ that
 is subharmonic, i.e. there exists a subharmonic function $u:\Omega\to\RE$ such that
     \[
     \lim_{x\to y_0\atop x\in\Omega} u(x)=0\, ,
     \]
     and
     \[
      \limsup_{x\to y\atop x\in\Omega} u(x)<0 \qquad  \text{ for all }  y\in\bar{\Omega}\backslash\{y_0\}\, .
     \]

 \item $\partial\Omega$ has a \emph{weak barrier} at every point $y_0\in\partial\Omega$  that
 is subharmonic, i.e. there exists a subharmonic function $u:\Omega\to\RE$ such that $u<0$ on $\Omega$ and
     \[
     \lim_{x\to y_0\atop x\in\Omega} u(x)=0\, .
     \]

 \item $\Omega$ admits an \emph{exhaustion function} that is negative and subharmonic,
        i.e.  there exists a non-constant function $\psi:\Omega\to\RE$ such that for any $c\in\RE$ the set $\{x\in\Omega:\psi(x)<c\}$ is relatively compact in $\Omega$. Furthermore, the exhaustion function should be negative and subharmonic.

 \item $\partial\Omega$ is equal to the \emph{Jensen boundary} w.r.t. the Jensen measures generated by the cone of functions that is continuous on $\bar\Omega$, and subharmonic on $\Omega$ (see Section~\ref{sec_prelim} for definitions).

 \end{enumerate}

\bigskip

\noindent The idea of a regular boundary point can be traced back to 1911 and 1912 with the works of Zaremba~\cite{zaremba} and Lebesgue~\cite{Lebesgue1}, respectively, when they constructed examples that exhibit the existence of irregular points. A decade after these examples Perron  introduced in 1923 the celebrated envelope construction $\operatorname{PWB}_{f}$ (see condition (1)). The work on $\operatorname{PWB}_{f}$ was later continued by Wiener~\cite{wiener1,wiener2,wiener3}, and  in our setting concluded by Brelot~\cite{brelot} in 1939. The notion of barrier goes further back in time; it can be found in the work of Poincar\'{e}~\cite{Poincare} from 1890. The implication (3) $\Rightarrow$ (1) is due to Bouligand~\cite{bouligand} who generalized a result of Lebesgue~\cite{Lebesgue2}. The equivalence with assertion (5) originates from the study of function algebras known as Choquet theory, which was developed in
the 50's and 60's by Bauer, Bishop, Choquet, de Leeuw, and others (see e.g.~\cite{edwards,Gamelin,GamelinSibony} and the references therein). For a beautiful treatise on Choquet theory we highly recommend~\cite{Lukes et al}.

\bigskip

 Inspired by the beauty of the equivalences in Theorem~A, analysts started to investigate these notions within the model introduced by Lelong~\cite{lelong} and Oka~\cite{oka} in 1942,  where subharmonic functions are changed to plurisubharmonic functions. The unit polydisc in $\C^n$, $n\geq 2$, shows that the notions of weak and strong barrier for plurisubharmonic functions are not equivalent. Instead we have Theorem~B and~Theorem~C below, where we  assume that $n\geq 2$. If $n=1$, then the two theorems become Theorem~A since
 subharmonic functions are then the same as plurisubharmonic functions.

 \bigskip

 \noindent \textbf{Theorem~B.}  Assume that $\Omega$ is a bounded domain in $\C^n$, $n\geq 2$. Then the following assertions are equivalent.

 \medskip

 \begin{enumerate}\itemsep2mm

 \item $\partial\Omega$ is \emph{B-regular} at every boundary point $z_0\in\partial\Omega$, in the sense that
 \[
\lim_{z\to z_0\atop z\in\Omega} \operatorname{PB}_f(z) =f(z_0)\, ,
 \]
for each continuous function $f:\partial\Omega\to \RE$. Here
\[
\operatorname{PB}_f(z)=\sup\Bigg\{v(z): v\in\mathcal{PSH}(\Omega),\; \varlimsup_{\zeta\rightarrow\xi \atop
\zeta\in\Omega}v(\zeta)\leq f(\xi)\, , \;\; \forall \xi\in\partial\Omega\Bigg\}\, .
\]
Here $\mathcal{PSH}(\Omega)$ is the space of plurisubharmonic functions defined on $\Omega$;

\item $\partial\Omega$ has a strong barrier at every point that is plurisubharmonic;

\item $\Omega$ admits an exhaustion function $\varphi$ that is negative, smooth, plurisubharmonic, and such that
    $\left(\varphi(z)-|z|^2\right)$ is plurisubharmonic.

\item $\partial\Omega$ is equal to the Jensen boundary w.r.t. the Jensen measures generated by the cone of functions that is continuous on $\bar\Omega$, and plurisubharmonic on $\Omega$.

 \end{enumerate}

 \bigskip

\noindent In 1959, Bremermann~\cite{bremermann} adopted the idea from assertion (1) in Theorem~A to pluripotential theory (see (1) in Theorem~B). He named his construction the
Perron-Carath\'{e}odory function after the articles~\cite{carat,perron}. The name did not survive the passage of time, and now it is known as the Perron-Bremermann envelope.  Drawing inspiration from Choquet theory, and its representing measures~\cite{Gamelin,GamelinSibony,rickart}, Sibony proved Theorem~B in the article~\cite{S}, which was published in 1987. There he also put these conditions in connection with Catlin's property $(P)$, and the
$\bar{\partial}$-Neumann problem. The last condition in assertion $(3)$ means that we have that
\[
\sum_{j,k=1}^{n}\frac{\partial^2\varphi}{\partial z_j\partial \bar{z}_k} \alpha_j\bar{\alpha}_k\geq |\alpha|^2\,, \;\text{ for all }\alpha\in\C^n\, .
\]
Hence, one can interpret $\varphi$ as being \emph{uniformly} strictly plurisubharmonic.

\bigskip

 \noindent \textbf{Theorem~C.} Assume that $\Omega$ is a bounded domain in $\C^n$, $n\geq 2$. Then the following assertions are equivalent.

 \medskip

 \begin{enumerate}\itemsep2mm

 \item $\Omega$ is \emph{hyperconvex} in the sense that it admits an exhaustion function that is negative and plurisubharmonic;

 \item $\partial\Omega$ has a weak barrier at every point that is plurisubharmonic;

 \item $\Omega$ admits an exhaustion function that is negative, smooth and strictly plurisubharmonic;

 \item  for every $z\in \partial \Omega$,  and every Jensen measure $\mu$, which is generated by the cone of functions that is continuous on $\bar\Omega$, and plurisubharmonic on $\Omega$, we have that $\mu$ is carried by $\partial \Omega$.

 \end{enumerate}

\bigskip

\noindent Historically, the notion of hyperconvexity was introduced by Stehl\'{e} in 1974 in connection with the Serre conjecture, and later in 1981 Kerzman and Rosay~\cite{KR} proved the equivalence of the three first assertions (see also~\cite{aytuna}). Kerzman and Rosay also considered the question of which pseudoconvex domains are hyperconvex. We shall not address this question here (see e.g. the introduction of~\cite{avelin} for an up-to-date account of this question). Carlehed et al.~\cite{CCW} showed in 1999 the equivalence between (1) and (4). In connection with Theorem~B and Theorem~C we would like to mention the inspiring article~\cite{Blocki_MA} written B\l ocki, the first part of which is a self-contained survey on plurisubharmonic barriers and exhaustion functions in complex domains.

\bigskip

 As we mentioned at the beginning of this expos\'{e} the purpose of this paper is to study the geometry of the corresponding notions $B$-regular and hyperconvex domains within the Caffarelli-Nirenberg-Spruck model. More precisely, in Theorem~\ref{bm-reg} we prove what degenerates into Theorem~B when $m=n$, and in Theorem~\ref{thm_mhx} we prove what is Theorem~C in the case $m=n$, except for the corresponding implication $(1)\Rightarrow (3)$. This we prove
in Section~\ref{sec_smooth} due to the different techniques used, and the length of that proof. In the case when $m=1$, our
Theorem~\ref{bm-reg} and Theorem~\ref{thm_mhx} (together with Theorem~\ref{smoothexh}) merge into Theorem~A above with $N=2n$.

\bigskip

This article is organized as follows. In Section~\ref{sec_prelim} we shall state the necessary definitions and some preliminaries needed for this paper, and then
in Section~\ref{sec_basic} we shall prove some basic facts of $m$-hyperconvex domains (Theorem~\ref{thm_prop}). From Section~\ref{sec_basic}, and Theorem~\ref{thm_prop} we would like
the reader to take special note of property (3). Up until now authors have defined $m$-hyperconvex domains to be bounded domains
that admit an exhaustion function that is negative, continuous, and $m$-subharmonic. We prove that the assumption of continuity
is superfluous. This result is also the starting point of the proof of Theorem~\ref{smoothexh}. In Section~\ref{sec_geometry} we prove Theorem~\ref{bm-reg} and Theorem~\ref{thm_mhx}, as mentioned above, which correspond to Theorem~B and Theorem~C, respectively. We end this paper by showing that every $m$-hyperconvex domain admits a smooth and strictly $m$-subharmonic exhaustion function (Theorem~\ref{smoothexh}; see implication $(1)\Rightarrow (3)$ in Theorem~C).

\bigskip

We end this introduction by highlighting an opportunity for future studies related to this paper. As convex analysis
and pluripotential theory lives in symbiosis, Trudinger and Wang~\cite{TW2} draw its inspiration from the work of Caffarelli et al., and in 1999 they presented a model that makes it possible to study the transition between convex analysis and potential theory. For further information see e.g~\cite{TW1, TW2, TW3, Wang}. As~\cite{WanWang} indicates, further studies of the geometric properties of what could be named \emph{$k$-convex domains} are of interest. We leave these questions to others.

\bigskip

We want to thank Urban Cegrell, Per-H\aa kan Lundow, and H\aa kan Persson for inspiring discussions related to this paper.

\section{Preliminaries}\label{sec_prelim}

In this section we shall present the necessary definitions and fundamental facts needed for the rest of this paper.
For further information related to potential theory see e.g.~\cite{armitage,doob,landkof}, and for more information about pluripotential theory see e.g.~\cite{demailly_bok,K}. We also want to mention the highly acclaimed book written by H\"ormander called ``\emph{Notions of convexity}''~\cite{hormander}. Abdullaev and Sadullaev~\cite{SA} have written an article that can be used as an introduction to the Caffarelli-Nirenberg-Spruck model, as well as Lu's doctoral thesis~\cite{L}. We would like to point out that $m$-subharmonic functions in the sense
of Caffarelli et al. is not equivalent of being subharmonic on $m$-dimensional hyperplanes in $\C^n$ studies by others
(see e.g.~\cite{Abdullaev1,Abdullaev2}). For other models in connection to plurisubharmonicity see e.g.~\cite{HarveyLawson2,HarveyLawson3,HarveyLawson4}.

Let $\Omega \subset \C^n$ be a bounded domain, $1\leq m\leq n$, and define $\mathbb C_{(1,1)}$ to be the set of $(1,1)$-forms with constant coefficients. With this notation we define
\[
\Gamma_m=\left\{\alpha\in \mathbb C_{(1,1)}: \alpha\wedge \beta^{n-1}\geq 0, \dots , \alpha^m\wedge \beta ^{n-m}\geq 0   \right\}\, ,
\]
where $\beta=dd^c|z|^2$ is the canonical K\"{a}hler form in $\C^n$.

\begin{definition}\label{m-sh} Assume that $\Omega \subset \C^n$ is a bounded domain, and let $u$ be a subharmonic function defined  $\Omega$. Then we say that $u$ is \emph{$m$-subharmonic}, $1\leq m\leq n$, if the following inequality holds
\[
dd^cu\wedge\alpha_1\wedge\dots\wedge\alpha_{m-1}\wedge\beta^{n-m}\geq 0\, ,
\]
in the sense of currents for all $\alpha_1,\ldots,\alpha_{m-1}\in \Gamma_m$. With $\mathcal{SH}_m(\Omega)$ we denote the set of all $m$-subharmonic functions defined on $\Omega$. We say that a function $u$ is \emph{strictly $m$-subharmonic} if it is $m$-subharmonic on $\Omega$, and for every $p \in \Omega$ there exists a constant $c_p >0$ such that $u(z)-c_p |z| ^2$ is $m$-subharmonic in a neighborhood of $p$.
\end{definition}
\begin{remark} From Definition~\ref{m-sh} it follows that
\[
\mathcal{PSH}=\mathcal{SH}_n \subset \cdots \subset \mathcal{SH}_1 =\mathcal{SH}\, .
\]
\end{remark}

\bigskip

In Theorem~\ref{thm_basicprop1} we give a list of well-known properties that $m$-subharmonic functions enjoy.

\bigskip

\begin{theorem} \label{thm_basicprop1} Assume that $\Omega \subset \C^n$ is a bounded domain, and $1\leq m\leq n$. Then we have that

\medskip

    \begin{enumerate}\itemsep2mm
    \item if $u,v\in \mathcal{SH}_m(\Omega)$, then $su+tv\in \mathcal{SH}_m(\Omega)$, for constants $s,t\geq 0$;

    \item if $u,v\in \mathcal{SH}_m(\Omega)$, then $\max\{u,v\}\in \mathcal{SH}_m(\Omega)$;

    \item if $\{u_{\alpha}\}$ is a locally uniformly bounded family of functions from $ \mathcal{SH}_m(\Omega)$, then the upper semicontinuous regularization
        \[
        \left(\sup_{\alpha}u_{\alpha}\right)^*
        \]
        defines a $m$-subharmonic function;

    \item if $\{u_j\}$ is a sequence of functions in $\mathcal{SH}_m(\Omega)$ such that $u_j \searrow u$ and there is a point $z \in \Omega$ such that $u(z) > - \infty$, then $u \in \mathcal{SH}_m(\Omega)$;

    \item if $u\in \mathcal{SH}_m(\Omega)$ and $\gamma:\mathbb R\to\mathbb R$ is a convex and nondecreasing function, then $\gamma\circ u\in\mathcal{SH}_m(\Omega)$;

    \item if $u\in \mathcal{SH}_m(\Omega)$, then the standard regularization given by the convolution $u\star \rho_{\varepsilon}$ is $m$-subharmonic in $\{z\in \Omega:\operatorname{dist}(z,\partial \Omega)>\varepsilon\}$. Here we have that
    \[
    \rho_{\varepsilon}=\varepsilon^{-2n}\rho\left(\frac z{\varepsilon}\right)\, ,
    \]
    $\rho:\mathbb R_+\to \mathbb R_+$ is a smooth function such that $\rho(z)=\rho(|z|)$ and
    \[
    \rho(t)=\begin{cases}
    \frac {C}{(1-t)^2}\exp\left({\frac {1}{t-1}}\right) & \text{ when } t\in [0,1]\\
    0 & \text{ when } t\in (1,\infty)\, ,
    \end{cases}
    \]
    where $C$ is a constant such that $\int_{\mathbb C^n}\rho(|z|^2)\beta^n=1$;

\item if $\omega\Subset\Omega$, $u\in \mathcal{SH}_m(\Omega)$, $v\in \mathcal{SH}_m(\omega)$, and $\varlimsup_{z\to w}v(z)\leq u(w)$ for all $w\in \partial \omega$, then the function defined by
\[
\varphi=\begin{cases}
u, \, \text { on } \, \Omega\setminus \omega\, ,\\
\max\{u,v\}, \; \text { on } \, \omega,
\end{cases}
\]
is $m$-subharmonic on $\Omega$;
   \end{enumerate}
\end{theorem}

We shall need several different envelope constructions. We have gathered their definitions and notations in Definition~\ref{def}.

\begin{definition}\label{def}
Assume that $\Omega \subset \C^n$ is a bounded domain, and $1\leq m\leq n$.

\begin{itemize}\itemsep2mm
\item[$a)$] For $f\in \mathcal C(\bar\Omega)$ we define
\[
\textbf {S}_f(z)=\sup\left\{v(z): v\in \mathcal {SH}_m(\Omega), v\leq f\right\}\, ,
\]
and similarly
\[
\textbf {S}^c_f(z)=\sup\left\{v(z): v\in \mathcal {SH}_m(\Omega)\cap\mathcal C(\bar\Omega), v\leq f\right\}\, .
\]
\item[$b)$] If instead $f\in \mathcal C(\partial \Omega)$, then we let
\[
\textbf {S}_f(z)=\sup\left\{v(z): v\in \mathcal {SH}_m(\Omega), v^*\leq f\, \text { on }\, \partial \Omega\right\}\, ,
\]
and
\[
\textbf {S}^c_f(z)=\sup\left\{v(z): v\in \mathcal {SH}_m(\Omega)\cap \mathcal C(\bar\Omega), v\leq f\, \text { on }\, \partial \Omega\right\}.
\]
\end{itemize}
\end{definition}

\begin{remark}
If $\Omega\subset\C^n$ $(\cong \RE^{2n})$ is a regular domain in the sense of Theorem~A, and if $f\in \mathcal C(\partial \Omega)$, then $\operatorname{PWB}_f$ (defined also in Theorem~A) is the unique harmonic function on $\Omega$, continuous on $\bar \Omega$, such that $\operatorname{PWB}_f=f$ on $\partial \Omega$. Therefore, we have that $\textbf{S}_f(z)=\textbf{S}_{\operatorname{PWB}_f}(z)$, and $\textbf {S}^c_f(z)=\textbf{S}^c_{\operatorname{PWB}_f}(z)$.
\end{remark}

In Definition~\ref{def2} we state the definition of relative extremal functions in our setting.

\begin{definition}\label{def2}
Assume that $E\Subset\Omega$ is an open subset such that $\Omega\setminus \bar E$ is a regular domain in the sense of Theorem~A. Then we make the following definitions
\[
\textbf {S}_E(z)=\sup\left\{v(z): v\in \mathcal {SH}_m(\Omega), v\leq -1\, \text { on }\, E, v\leq 0 \right\}\, ,
\]
and
\[
\textbf {S}^c_E(z)=\sup\left\{v(z): v\in \mathcal {SH}_m(\Omega)\cap \mathcal C(\bar \Omega), v\leq -1\, \text { on }\, E, v\leq 0 \right\}.
\]
\end{definition}
\begin{remark}
From well-known potential theory we have that if $h_E$ is the unique harmonic function defined on $\Omega\setminus \bar E$, continuous on $\bar \Omega\setminus E$, $h_E=0$ on $\partial \Omega$, $h_E=-1$ on $\partial E$, and if we set
\[
\textbf{H}_E(z)=\begin{cases}
h_E(z) & \text{ if } z\in \bar \Omega \setminus E \\
-1& \text{ if } z\in E\, ,
\end{cases}
\]
then we have that $\textbf {S}_E(z)=\textbf{S}_{\textbf{H}_E}(z)$ and $\textbf {S}^c_E(z)=\textbf{S}^c_{\textbf{H}_E}(z)$.
\end{remark}

B\l ocki's generalization of Walsh's celebrated theorem~\cite{walsh}, and an immediate consequence will be needed as well.

\begin{theorem}\label{walsh}
Let $\Omega$ be a bounded domain in $\mathbb C^n$, and let $f\in \mathcal C(\bar \Omega)$.
If for all $w\in \partial \Omega$ we have that $\lim_{z\to w}\textbf{S}_f(z)=f(w)$, then $\textbf{S}_f\in \mathcal{SH}_m(\Omega)\cap \mathcal C(\bar\Omega)$.
\end{theorem}
\begin{proof} See Proposition~3.2 in~\cite{Blocki_weak}.

\end{proof}

A direct consequence of Theorem~\ref{walsh} is the following.

\begin{corollary}\label{walsh2}
Let $\Omega$ be a bounded domain in $\mathbb C^n$, and let $f\in \mathcal C(\bar \Omega)$.
If for all $w\in \partial \Omega$ we have that $\lim_{z\to w}\textbf {S}^c_f(z)=f(w)$, then $\textbf {S}^c_f=\textbf {S}_f\in \mathcal{SH}_m(\Omega)\cap \mathcal C(\bar\Omega)$.
\end{corollary}
\begin{proof} First note that
\[
\textbf {S}^c_f\leq \textbf {S}_f\leq f\, .
\]
Therefore,  if
\[
\lim_{z\to w}\textbf {S}^c_f(z)=f(w)\, ,
\]
holds for all $w\in \partial \Omega$, then
\[
\lim_{z\to w}\textbf {S}_f(z)=f(w)\, .
\]
Hence, by Theorem~\ref{walsh} we get that $\textbf {S}_f\in \mathcal{SH}_m(\Omega)\cap \mathcal C(\bar\Omega)$, which
gives us that $\textbf {S}_f\leq \textbf {S}^c_f$. Thus, $\textbf {S}_f=\textbf {S}^c_f$.
\end{proof}

In Section~\ref{sec_geometry}, we shall make use of techniques from Choquet theory, in particular Jensen measures w.r.t.
the cone $\mathcal{SH}_m(\Omega) \cap \mathcal C(\bar \Omega)$ of continuous functions. This is possible since $\mathcal{SH}_m(\Omega) \cap \mathcal C(\bar \Omega)$ contains the constant functions and separates points in $\mathcal C(\bar \Omega)$. Our inspiration can be traced back to the works mentioned in the introduction, but maybe more to~\cite{CCW} and~\cite{HP}.

\begin{definition}\label{def_JzmO} Let $\Omega$ be a bounded domain in $\C^n$, and let $\mu$ be a non-negative regular Borel measure defined on $\bar{\Omega}$. We say that $\mu$ is a \emph{Jensen measure with barycenter} $z_0\in\bar{\Omega}$ \emph{w.r.t.} $\mathcal{SH}_m(\Omega) \cap \mathcal C(\bar \Omega)$ if
  \[
  u(z_0)\leq \int_{\bar{\Omega}} u \, d\mu \qquad\qquad \text{for all } u  \in \mathcal{SH}_m(\Omega) \cap \mathcal C(\bar \Omega)\, .
  \]
The set of such measures will be denoted by $\mathcal{J}_{z_0}^m$. Furthermore, the  \emph{Jensen boundary} w.r.t. $\mathcal{J}_{z_0}^m$ is defined as
\[
\partial_{\mathcal{J}^m}=\left\{z\in\bar{\Omega}: \mathcal{J}_{z}^m=\{\delta_z\} \right\}\, .
\]
\end{definition}
\begin{remark} The Jensen boundary is another name for the Choquet boundary w.r.t. a given class of Jensen measures. For further information see e.g.~\cite{brelot2,Lukes et al}.
\end{remark}
\begin{remark} There are many different spaces of Jensen measures introduced throughout the literature. Caution is advised.
\end{remark}

The most important tool in working with Jensen measures is the Edwards' duality theorem that origins from~\cite{edwards}. We only need a special case formulated
in Theorem~\ref{eq_edwardversion2}. For a proof, and a discussion, of Edwards' theorem see~\cite{W} (see also~\cite{ColeRansford1,ColeRansford2,Ransford}).

\begin{theorem}[Edwards' Theorem]\label{eq_edwardversion2}
Let $\Omega$ be a bounded domain in $\C^n$, and let $g$ be a real-valued  lower semicontinuous function defined on $\bar \Omega$. Then for every $z\in\bar{\Omega}$ we have that
\[
\textbf{S}_g^c(z)=\sup\{v(z):v \in \mathcal{SH}_m(\Omega) \cap \mathcal C(\bar \Omega), v \leq g \}=\inf\left \{\int g \, d\mu : \mu \in \mathcal{J}_z^m\right\}\, .
\]
\end{theorem}

We end this section with a convergence result.

\begin{theorem} \label{thm_convmeasures}
Assume that $\Omega$ is a domain in $\C^n$, and let $\{z_n\} \subset \bar \Omega$ be a sequence of points converging to $z\in \bar \Omega$. Furthermore, for each $n$, let $\mu_n \in \mathcal{J}_{z_n}^m$. Then there exists a subsequence $\{\mu_{n_j}\}$, and a measure $\mu \in \mathcal{J}_z^m$ such that $\{\mu_{n_j}\}$ converges in the weak-$^\ast$ topology
to $\mu$.
\end{theorem}
\begin{proof} The Banach-Alaoglu theorem says that the space of probability measures defined on $\bar\Omega$ is compact when equipped with the weak-$^\ast$ topology. This means that there is a subsequence $\{\mu_{n_j}\}$ that converges to a probability measure $\mu$. It remains to show that $\mu \in \mathcal{J}_z^m$. Take $u \in \mathcal{SH}_m(\Omega)\cap\mathcal C(\bar \Omega)$ then
\[
\int u \, d\mu= \lim_j \int u \, d\mu_{n_j} \geq \lim_ju(z_j)=u(z),
\]
hence $\mu \in \mathcal{J}_z^m$.
\end{proof}

\section{Basic properties of $m$-hyperconvex domains}\label{sec_basic} 

The aim of this section is to introduce $m$-hyperconvex domains (Definition~\ref{def_mhx}) within the Caffarelli-Nirenberg-Spruck model, and prove Theorem~\ref{thm_prop} . If  $m=1$, then the notion
will be the same as regular domains (see assertion (4) in Theorem~A in the introduction), and if $m=n$ then it is the same as hyperconvex domains (see
(1) in Theorem~C).

\begin{definition}\label{def_mhx}
  Let $\Omega$ be a bounded domain in $\C^n$. We say that $\Omega$ is \emph{$m$-hyperconvex} if it admits an exhaustion function that is negative and $m$-subharmonic.
\end{definition}

Traditionally, in pluripotential theory the exhaustion functions are assumed to be bounded. That assumption is obviously superfluous in Definition~\ref{def_mhx}. Even though it should be mentioned once again that up until now authors have defined $m$-hyperconvex domains to be bounded domains
that admit an exhaustion function that is negative, \emph{continuous}, and $m$-subharmonic. We prove below in Theorem~\ref{thm_prop}
that the assumption of continuity is not necessary. Before continuing with Theorem~\ref{thm_prop} let us demonstrate the concept of $m$-hyperconvexity in the
following two examples. Example~\ref{ex1} demonstrate that Hartog's triangle is $1$-hyperconvex, but not $2$-hyperconvex.

\begin{example}\label{ex1}
The Hartog's triangle $\Omega=\{(z,w) \in \C^2: |z | < | w |  <1\}$ is an example of a domain that is not hyperconvex (Proposition 1 in~\cite{DiederichFornaess}), i.e. it is not $2$-hyperconvex, but it is a regular domain, i.e. it is $1$-hyperconvex. It is easy to see that
\[
\varphi(z,w)=\max\big\{\log |w|,|z|^2-|w|^2\big\}\, .
\]
is a negative, subharmonic ($1$-subharmonic) exhaustion function for $\Omega$.
\hfill{$\Box$}
\end{example}

In Example~\ref{ex2} we construct a domain in $\mathbb{C}^3$ that is $2$-hyperconvex, but not $3$-hyperconvex.

\begin{example}\label{ex2}
For a given integer $1\leq k\leq n$, let $\varphi_k$ be the function defined on $\mathbb{C}^n$ by
\[
\varphi_k(z_1,\ldots,z_n)=|z_1|^2+\ldots+|z_{n-1}|^2+\left(1-\frac nk\right)|z_n|^2\, .
\]
Then we have that $\varphi_k$ is $m$-subharmonic function if, and only if, $m\leq k$. Let us now consider the
following domain:
\[
\Omega_k=\left\{(z_1,\dots,z_n)\in \mathbb C^n: |z_1|<1, \dots, |z_n|<1, \varphi_k(z)<1 \right\}\, .
\]
This construction yields that $\Omega_k$ is a balanced Reinhardt domain that is not pseudoconvex (see e.g. Theorem~1.11.13 in~\cite{jarnicki_pflug}). Furthermore, we have that $\Omega_k$ is $k$-hyperconvex, since
\[
u(z_1,\ldots,z_n)=\max\{|z_1|,\dots,|z_n|, \varphi_k(z)\}-1\,
\]
is a $k$-subharmonic exhaustion function. In particular, we get that for $n=3$, and $k=2$, the domain $\Omega_2\subset \mathbb C^3$ is $2$-hyperconvex but not $3$-hyperconvex.
\hfill{$\Box$}
\end{example}

The aim of this section is to prove the following theorem, especially property (3).

\begin{theorem}\label{thm_prop} Assume that $\Omega$, $\Omega_1$, and $\Omega_2$ are bounded $m$-hyperconvex domains in $\C^n$, $n\geq 2$, $1\leq m\leq n$. Then we have the following.
\begin{enumerate}\itemsep2mm

\item  If $\Omega_1 \cap \Omega_2$ is connected, then the domain $\Omega_1 \cap \Omega_2$ is $m$-hyperconvex in $\C^n$.

\item  The domain $\Omega_1 \times \Omega_2$  is $m$-hyperconvex in $\C^{2n}$.

\item The domain $\Omega$ admits a negative exhaustion function that is strictly $m$-subharmonic on $\Omega$, and continuous on $\bar \Omega$.

\item If $\Omega$ is a priori only a bounded domain in $\mathbb C^n$ such that for every $z\in \partial \Omega$ there exists a neighborhood $U_z$ such that $\Omega\cap U_z$ is $m$-hyperconvex, then $\Omega$ is $m$-hyperconvex.

\end{enumerate}
\end{theorem}
\begin{proof}

\noindent\emph{Part (1)}  For each $j=1,2$, assume that $\psi_j \in \mathcal{SH}_m(\Omega_j)$ is a negative exhaustion function for the $m$-hyperconvex domain  $\Omega_j$, $j=1,2$. Then $\max\{\psi_1,\psi_2\} \in \mathcal{SH}_m(\Omega_1\cap \Omega_2)$ is  a negative exhaustion function for $\Omega_1 \cap \Omega_2$.
Thus, $\Omega_1 \cap \Omega_2$ is $m$-hyperconvex in $\C^n$.

\bigskip

\noindent\emph{Part (2)} This part is concluded by defining a negative exhaustion function by
        \[
           \psi(z_1,z_2)=\max\{\psi_1(z_1),\psi_2(z_2)\} \in \mathcal{SH}_m(\Omega_1\times \Omega_2)\, .
        \]

\bigskip

\noindent\emph{Part (3)} The proof of this part is inspired by~\cite{C}. First we shall prove that there exists a negative and continuous exhaustion function. We know that $\Omega$ always admits a bounded, negative, exhaustion function $\varphi \in \mathcal{SH}_m(\Omega)$. Fix $w\in \Omega$ and $r>0$ such that
$B(w,r)\Subset\Omega$, and note that there exists a constant $M>0$ such that
\[
M\varphi \leq \textbf{H}_{B(w,r)}
\]
(the definition of $\textbf{H}_{B(w,r)}$ is in the remark after Definition~\ref{def2}). This construction implies that
\[
0=\lim_{z\to \partial \Omega}M\varphi(z)\leq \lim_{z\to \partial \Omega}\textbf{S}_{\textbf{H}_{B(w,r)}}(z)\leq \lim_{z\to \partial \Omega}\textbf{H}_{B(w,r)}(z)=0\, .
\]
Thanks to the generalized Walsh theorem (Theorem~\ref{walsh}) we have that
\[
\textbf{S}_{\textbf{H}_{B(w,r)}}=\textbf{S}_{B(w,r)}\in \mathcal{SH}_m(\Omega)\cap \mathcal C(\bar\Omega)\, ,
\]
and that $\textbf{S}_{B(w,r)}$ is a continuous exhaustion function.

Next, we shall construct a continuous \emph{strictly} $m$-subharmonic exhaustion function for $\Omega$. From the first part of this theorem we know that there is a negative and continuous exhaustion function $u \in \mathcal{SH}_m(\Omega)\cap \mathcal C(\bar \Omega)$  for $\Omega$.  Choose $M>0$ such that $|z|^2-M \leq -1$ on $\Omega$, and define
\[
\psi_j(z)=\max\left \{u(z),\frac{|z|^2-M}{j}\right\}\, .
\]
Then $\psi_j \in \mathcal{SH}_m(\Omega)\cap \mathcal C(\bar \Omega)$, $\psi_j|_{\partial \Omega}=0$, and $\psi_j <0$ on $\Omega$. If we now let
\[
a_j=\frac{1}{2^j}\frac{1}{\max\{\sup(-\psi_j),1\}}, \qquad \text{and} \qquad\psi=\sum_{j=1}^{\infty}a_j\psi_j\, ,
\]
then $\psi^k=\sum_{j=1}^{k}a_j\psi_j$ defines a decreasing sequence of continuous $m$-subharmonic functions on defined $\Omega$. We can conclude that $\psi \in \mathcal{SH}_m(\Omega)$, since $\psi(z)> -\infty$  for $z \in \Omega$. The continuity of $\psi$ is obtained by the Weierstrass $M$-test. To see that $\psi$ is strictly $m$-subharmonic, note that if $\omega \Subset \Omega$, then there exists an index $j_\omega$ such that on $\omega$ we have that
\[
\psi_j=\frac{|z|^2-M}{j} \qquad \text{ for all } j>j_\omega\, .
\]
This gives us that
\[
\psi=\sum_{j=1}^{j_\omega}a_j \psi_j+\sum_{j=j_\omega+1}^{\infty}a_j \frac{|z|^2-M}{j}\, .
\]
Since $\frac{|z|^2-M}{j}$ is strictly plurisubharmonic, and therefore strictly $m$-subharmonic, we have that $\psi$ is strictly $m$-subharmonic on $\Omega$. Finally, $\psi$ is an exhaustion function for $\Omega$, since  $\psi_j|_{\partial \Omega}=0$ for all $j$.

\bigskip

\noindent\emph{Part (4)} The idea of the proof of this part is from~\cite{Blocki_LN}. By the assumption there are neighborhoods $U_{z_1},\dots,U_{z_N}$ such that $\partial \Omega\subset \bigcup_{j=1}^NU_{z_j}$, and each $U_{z_j}\cap \Omega$ is $m$-hyperconvex. Let $u_j:\Omega\to [-1,0]$ be a negative and continuous $m$-subharmonic exhaustion function for
$U_{z_j}\cap \Omega$. Let $V_j\Subset U_{z_j}$ be such that  $\partial \Omega\subset \bigcup_{j=1}^NV_{j}$. For $x<0$, we then define the following continuous functions
\[
\begin{aligned}
\beta(x)&=\max\left\{u_j(z): z\in \bar V_j\cap \Omega, j=1,\dots, N, \operatorname{dist}(z,\partial \Omega)\leq -x  \right\}, \\
\alpha(x)&=\min\left\{u_j(z): z\in \bar V_j\cap \Omega, j=1,\dots, N, \operatorname{dist}(z,\partial \Omega)\leq -x  \right\}\, .
\end{aligned}
\]
From these definitions it follows that  $\alpha\leq \beta$,  and $\lim_{x\to 0^-}\alpha(x)=0$. Therefore, there exists a convex, increasing function $\chi:(-\infty,0)\to(0,\infty)$
such that $\lim_{x\to 0^-}\chi(x)=\infty$, and $\chi\circ \beta\leq \chi\circ\alpha+1$ (see e.g. Lemma A2.4. in~\cite{Blocki_LN}). Hence,
\[
 |\chi\circ u_j-\chi\circ u_k|\leq 1 \, \text { on } \, V_j\cap V_k\cap \Omega\, .
\]
For any $\varepsilon>0$ we have that
\begin{equation}\label{in1}
|\chi(u_j(z)-\varepsilon)-\chi(u_k(z)-\varepsilon)|\leq 1 \, \text { for } \, z\in V_j\cap V_k\cap \Omega\, ,
\end{equation}
since $\chi $ is an increasing and convex function. Next, let $V'_j\Subset V_j$, $j=1,\dots,N$, be such that $\bar\Omega\setminus V\subset \bigcup_{j=1}^NV_j'$, for some open set $V\Subset \Omega$. For each $j$, take a smooth function $\varphi_j$  such that
$\operatorname{supp}(\varphi_j)\subset V_j$, $0\leq \varphi_j\leq 1$, and $\varphi_j=1$ on a neighborhood of $\bar V_j'$. Furthermore, there are constants $M_1,M_2>0$ such that $|z|^2-M_1\leq 0$ on $\Omega$, and such that the functions $\varphi_j+M_2(|z|^2-M_1)$ are
$m$-subharmonic for $j=1,\dots,N$. Let us define
\[
v_{j,\varepsilon}(z)=\chi(u_j(z)-\varepsilon)+\varphi_j(z)-1+M_2(|z|^2-M_1)\, .
\]
From~(\ref{in1}) it then follows that
\begin{equation}\label{in2}
v_{j,\varepsilon}\leq v_{k,\varepsilon} \, \text { on a neigborhood of } \, \partial V_j\cap \bar V_k'\cap \Omega\, .
\end{equation}
Take yet another constant $c$ such that
\begin{equation}\label{in3}
\sup\left \{u_j(z): z\in V\cap V_j, j=1,\dots N\right \}\, < c<0\, ,
\end{equation}
and define
\[
v_{\varepsilon}(z)=\max\left \{ v_{j,\varepsilon}(z),\chi(c)-1+M_2(|z|^2-M_1)\right\}\, .
\]
By (\ref{in2}), and (\ref{in3}), it follows that $v_{\varepsilon}$ is a well-defined $m$-subharmonic function defined on $\Omega$. Finally note that, for $c<-\varepsilon$, the following function
\[
\psi_{\varepsilon}(z)=\frac {v_{\varepsilon}(z)}{\chi(-\varepsilon)}-1
\]
is $m$-subharmonic, and $\psi_{\varepsilon}\leq 0$ on $\Omega$. For $z\in \partial \Omega$,  we have that
\[
u_j(z)=0\qquad \text{ and }\qquad  \varphi_j(z)=1\, ,
\]
hence
\begin{equation}\label{in4}
\psi_{\varepsilon}(z)\geq \frac {v_{j,\varepsilon}(z)}{\chi(-\varepsilon)}-1=\frac {\chi(-\varepsilon)+M_2(|z|^2-M_1)}{\chi(-\varepsilon)}-1\geq-\frac {M_1M_2}{\chi(-\varepsilon)}.
\end{equation}
In addition, it holds that
\begin{equation}\label{in5}
\psi_{\varepsilon}(z)\leq \frac {\chi(c)-1}{\chi(-\varepsilon)}-1, \, z\in V\setminus \bigcup_{j=1}^NV_j\, .
\end{equation}
Now fix a ball $B(z,r)\Subset  V\setminus \bigcup_{j=1}^NV_j$. From (\ref{in4}), (\ref{in5}), and the fact that
\[
\lim_{x\to 0^-}\chi(x)=\infty
\]
we have that
\[
(\sup_{\varepsilon}\psi_{\varepsilon})^*\leq \textbf{S}_{B(z,r)}
\]
(see Definition \ref{def2}). Thus,
\[
\lim_{\xi\to \partial \Omega}\textbf{S}_{B(z,r)}(\xi)=0\, .
\]
Theorem~\ref{walsh} (generalized Walsh's theorem) gives us that
\[
\textbf{S}_{B(z,r)}\in
 \mathcal {SH}_m(\Omega)\cap\mathcal C(\bar \Omega)\, ,
 \]
and that $\textbf{S}_{B(z,r)}$ is the desired exhaustion function for $\Omega$. This ends the proof of Part (4), and this theorem.
\end{proof}

\section{The geometry of $m$-regular domains}\label{sec_geometry}

In this section, we shall investigate the geometry of the corresponding notions of $B$-regular and hyperconvex domains within the Caffarelli-Nirenberg-Spruck model. More precisely, in Theorem~\ref{bm-reg} we prove what degenerates into Theorem~B when $m=n$, and in Theorem~\ref{thm_mhx} we prove what is Theorem~C in the case $m=n$.

\begin{theorem}\label{thm_mhx} Assume that $\Omega$ is a bounded domain in $\C^n$, $n\geq 2$, $1\leq m\leq n$. Then the following assertions are equivalent.

 \medskip

 \begin{enumerate}\itemsep2mm

 \item $\Omega$ is $m$-hyperconvex in the sense of Definition~\ref{def_mhx};

 \item $\partial\Omega$ has a weak barrier at every point that is $m$-subharmonic;

 \item $\Omega$ admits an exhaustion function that is negative, smooth and strictly $m$-subharmonic;

 \item  for every $z\in \partial \Omega$,  and every $\mu\in\mathcal{J}_z^m$, we have that $\operatorname{supp} (\mu) \subseteq \partial \Omega$.
 \end{enumerate}
 \end{theorem}
\begin{proof} The implications $(1)\Rightarrow(2)$, and $(3)\Rightarrow(1)$ are trivial. The implication $(1)\Rightarrow(3)$ is postponed to Theorem~\ref{smoothexh} in Section~\ref{sec_smooth}.

\medskip

$(2)\Rightarrow(1):$ Let $w\in \Omega$ and $r>0$ be such that the ball $B(w,r)\Subset \Omega$. Then by assumption we have that for every $z\in \partial \Omega$
there exists a weak barrier $u_z$ at $z$ that is $m$-subharmonic. Since there exists a constant $M_z>0$ such that
\[
M_zu_z\leq \textbf{S}_{B(w,r)}
\]
it follows that
\[
\lim_{\xi\to \partial \Omega}\textbf{S}_{B(w,r)}(\xi)=0\, .
\]
Thanks to the generalized Walsh theorem (Theorem~\ref{walsh}) we know that $\textbf{S}_{B(w,r)}\in \mathcal {SH}_m(\Omega)\cap \mathcal C(\bar \Omega)$. Hence, $\textbf{S}_{B(w,r)}$ is an exhaustion function for $\Omega$.

\medskip

 $(1)\Rightarrow(4):$ Assume that $\Omega$ is $m$-hyperconvex, and that $u \in \mathcal{SH}_m(\Omega)\cap \mathcal C(\bar \Omega)$ is an exhaustion function for $\Omega$. If $z \in \partial \Omega$, and $\mu \in \mathcal{J}_z^m$, then
  \[
  0=u(z)\leq \int u \, d\mu \leq 0\, .
  \]
  This implies that $\operatorname{supp} (\mu) \subseteq \partial \Omega$, since $u <0$ on $\Omega$.

 \medskip

$(4)\Rightarrow(1):$ Suppose that $\operatorname{supp}(\mu) \subset \partial \Omega$ for all $\mu \in \mathcal{J}_z^m$,  $z \in \partial \Omega$. Let $w\in \Omega$, $r>0$, be such that the ball $B(w,r)\Subset \Omega$, and let
\[
\textbf{S}^c_{B(w,r)}(z)= \sup\{\varphi(z): \varphi \in \mathcal{SH}_m(\Omega)\cap \mathcal C(\bar \Omega), \varphi \leq 0, \varphi \leq -1 \ \text{on} \ B(w,r)\}\, .
\]
From Edwards' theorem (Theorem \ref{eq_edwardversion2}) it follows that
\[
\textbf{S}^c_{B(w,r)}(z)=\inf\left\{\int -\chi_{B(w,r)} \, d\mu : \mu \in \mathcal{J}_z^m\right\}=-\sup\left\{\mu(B(w,r)): \mu \in \mathcal{J}_z^m\right\}.
\]
We shall now prove that
\[
\lim_{\xi\to \partial \Omega}\textbf{S}^c_{B(w,r)}(\xi)=0\, ,
\]
and this shall be done with a proof by contradiction. Assume the contrary, i.e. that there is a point $z \in \partial \Omega$ such that
\[
\varliminf_{\xi \rightarrow z}\textbf{S}^c_{B(w,r)}(\xi)<0\, .
\]
Then we can find a sequence $\{z_n\}$, that converges to $z$, and
\[
\textbf{S}^c_{B(w,r)}(z_n)<-\varepsilon\qquad \text{ for every } n\, .
\]
We can find corresponding measures $\mu_n \in \mathcal{J}_{z_n}^m$ such that $\mu_n(B(w,r))>\varepsilon$. By passing to a subsequence, Theorem~\ref{thm_convmeasures} gives us that we can assume that $\mu_n$ converges in the weak-$^\ast$ topology to a measure $\mu \in \mathcal{J}_z^m$. Lemma 2.3 in \cite{CCW}, implies then that
\[
\mu(\overline {B(w,r)})=\int \chi_{\overline {B(w,r)}} \, d\mu \geq \varlimsup_{n \rightarrow \infty} \int \chi_{\overline {B(w,r)}} \, d\mu_n =\varlimsup_{n \rightarrow \infty} \mu_n(\overline {B(w,r)}) >\varepsilon\geq 0\, .
\]
This contradicts the assumption that $\mu \in \mathcal{J}_z^m$ only has support on the boundary. Hence, Corollary~\ref{walsh2} gives us that
\[
\textbf{S}^c_{B(w,r)} \in \mathcal {SH}_m(\Omega)\cap \mathcal C(\bar \Omega)\, ,
\]
and that $\textbf{S}^c_{B(w,r)}$ is an exhaustion function for $\Omega$. Thus, $\Omega$ is $m$-hyperconvex.

\end{proof}

Before we can start with the proof of Theorem~\ref{bm-reg} we need the following corollary.

\begin{corollary}\label{cor3}
Let $\Omega$ be a bounded $m$-hyperconvex domain in $\C^n$, and let $f \in \mathcal C(\partial \Omega)$. Then there exists a function $u \in \mathcal{SH}_m(\Omega) \cap \mathcal  C(\bar \Omega)$ such that $u=f$ on $\partial \Omega$ if, and only if,
\[
f(z)=\inf\left\{\int f \, d\mu: \mu \in \mathcal{J}_z^m\right\}\qquad \text{ for all } z \in \partial \Omega\, .
\]
\end{corollary}
\begin{proof}
Assume that $f \in \mathcal C(\partial \Omega)$, and that $u \in \mathcal{SH}_m(\Omega)\cap \mathcal C(\bar \Omega)$ is such that $u=f$ on $\partial \Omega$. Let $z \in \partial \Omega$, and $\mu \in \mathcal{J}_z^m$, then we have that
\[
f(z)=u(z) \leq \int u \, d\mu\, ,
\]
which, together with Theorem~\ref{thm_mhx}, imply that
\[
f(z) \leq \inf\left\{\int u \, d\mu: \mu \in \mathcal{J}_z^m\right\}=\inf\left\{\int f \, d\mu: \mu \in \mathcal{J}_z^m\right\}\, .
\]
Since $\delta_z \in \mathcal{J}_z^m$ we have that
\[
\inf\left\{\int f \, d\mu: \mu \in \mathcal{J}_z^m\right\} \leq \int f \, d \delta_z=f(z)\, .
\]
Hence,
\[
f(z)=\inf\left\{\int u \, d\mu: \mu \in \mathcal{J}_z^m\right\} \qquad \text{ for } z \in \partial \Omega\, .
\]
Conversely, extend $f$ to a continuous function on $\bar \Omega$ (for instance one can take $\operatorname{PWB}_f$, which was defined in Theorem~A in the introduction) and for simplicity denote it also by $f$. Since $\Omega$ is a $m$-hyperconvex domain then by Theorem~\ref{thm_mhx}
for any $z\in \partial \Omega$ and any $\mu\in \mathcal{J}_z^m$ holds $\operatorname{supp} (\mu) \subseteq \partial \Omega$, so we have
\[
f(z)=\inf\left\{\int f \, d\mu : \mu \in \mathcal{J}_z^m\right\}\qquad \text{ for all } z \in \partial \Omega\, .
\]
Edwards' theorem (Theorem~\ref{eq_edwardversion2}) gives us now that
\[
\textbf{S}^c_f(z)=\inf\left\{\int f \, d\mu: \mu \in \mathcal{J}_z^m\right\}\, ,
\]
and therefore $\textbf{S}^c_f=f$ on $\partial \Omega$. To conclude this proof we shall prove that for $z\in \partial \Omega$ it holds that
\[
\lim_{\xi\to z}\textbf{S}^c_f(\xi)=f(z)\, .
\]
We shall argue by contradiction. Assume that
\[
\varliminf_{\xi \rightarrow z}\textbf{S}^c_f(\xi) < f(z)\qquad \text{ for some } z \in \partial \Omega\, .
\]

Then we can find an $\varepsilon >0$, and a sequence $\xi_j \rightarrow z$ such that
\[
\textbf{S}^c_f(\xi_j)< f(z)-\varepsilon\qquad\text{ for every } j\, .
\]
Since, for every $j$, we have that
\[
\textbf{S}^c_f(\xi_j)=\inf\left\{\int f \, d\mu: \mu \in \mathcal{J}_{\xi_j}^m\right\}
\]
there are measures $\mu_j \in \mathcal{J}_{\xi_j}^m$ such that
\[
\int f \, d\mu_j < f(z) - \varepsilon\, .
\]
By passing to a subsequence, and using Theorem \ref{thm_convmeasures}, we can assume that $\mu_j$ converges in the weak-$^\ast$ topology to some $\mu \in \mathcal{J}_z^m$. Hence,
\[
\int f \, d\mu=\lim_j \int f \, d\mu_j < f(z)-\varepsilon\, .
\]
This contradicts the assumption that
\[
f(z)=\inf\left\{\int f \, d\mu : \mu \in \mathcal{J}_z^m\right\}\, .
\]
Therefore, by Corollary~\ref{walsh2}, $\textbf{S}^c_f \in \mathcal{SH}_m(\Omega) \cap\mathcal C(\bar \Omega)$, and the proof is finished.
\end{proof}
\begin{remark}
If $\Omega$ is a bounded domain that is not necessarily $m$-hyperconvex, then we have a similar result as in Corollary \ref{cor3} namely that there exists a function $u \in \mathcal{SH}_m(\Omega) \cap \mathcal C(\bar \Omega)$ such that $u=f$ on $\bar \Omega$ if, and only if, there exists a continuous extension $\varphi$ of $f$ to $\bar \Omega$ such that
\[
\varphi(z)=\inf\left\{\int \varphi \, d\mu: \mu \in \mathcal{J}_z^m\right\}\, .
\]
\end{remark}

We end this section by proving Theorem~\ref{bm-reg}, and it's immediate consequence. We have in Theorem~\ref{bm-reg} decided to deviate from the
notation from Definition~\ref{def}. This to simplify the comparison with Theorem~B in the introduction.

\begin{theorem} \label{bm-reg}
Assume that $\Omega$ is a bounded domain in $\C^n$, $n\geq 2$, $1\leq m\leq n$. Then the following assertions are equivalent.

\medskip

 \begin{enumerate}\itemsep2mm

\item  $\partial\Omega$ is $B_m$-regular at every boundary point $z_0\in\partial\Omega$, in the sense that
 \[
\lim_{z\to z_0\atop z\in\Omega} \operatorname{PB}^m_f(z) =f(z_0)\, ,
 \]
for each continuous function $f:\partial\Omega\to \RE$. Here
\[
\operatorname{PB}^m_f(z)=\sup\Bigg\{v(z): v\in\mathcal{SH}_m(\Omega),\; \varlimsup_{\zeta\rightarrow\xi \atop \zeta\in\Omega}v(\zeta)\leq f(\xi)\, , \;\; \forall \xi\in\partial\Omega\Bigg\}\, .
\]

\item $\partial\Omega$ has a strong barrier at every point that is $m$-subharmonic;

\item $\Omega$ admits an exhaustion function $\varphi$ that is negative, smooth, $m$-subharmonic, and such that
    \[
   \left(\varphi(z)-|z|^2\right)\in \mathcal{SH}_m(\Omega)\, ;
    \]

 \item  $\partial\Omega=\partial_{\mathcal{J}_z^m}$ in the sense of Definition~\ref{def_JzmO}.

  \end{enumerate}
 \end{theorem}
\begin{proof}
$(1)\Rightarrow(2):$ Fix $z\in \partial \Omega$, and let $f$ be a continuous function on $\partial \Omega$ such that $f(z)=0$ and $f(\xi)<0$ for $\xi \neq z$.
Then $\operatorname{PB}^m_f$ is a strong barrier at $z$.

\medskip

$(2)\Rightarrow(1):$ Let $f\in \mathcal C(\partial \Omega)$. Then the upper semicontinuous regularization $(\operatorname{PB}^m_f)^*$ is $m$-subharmonic, and by
the generalized Walsh theorem (Theorem~\ref{walsh}) it is sufficient to show that
\[
\lim_{\xi\to \partial \Omega}\operatorname{PB}^m_f=f
\]
to obtain that $\operatorname{PB}^m_f\in \mathcal {SH}_m(\Omega)\cap\mathcal C(\bar\Omega)$. Fix $w\in \partial \Omega$, and $\varepsilon >0$. Let $u_w\in \mathcal {SH}_m(\Omega)$ be a strong barrier at $w$ that is $m$-subharmonic. Then there exists a constant $M>0$ such that
\[
f(w)+Mu_w^*-\varepsilon \leq f\, , \qquad \text{ on }\partial \Omega\, ,
\]
and therefore we have that $f(w)+Mu_w-\varepsilon\leq \operatorname{PB}^m_f$. This gives us that
\[
\varliminf_{\xi \to w}\operatorname{PB}^m_f(\xi)\geq f(w)-\varepsilon\, ,
\]
and finally $\lim_{\xi\to w}\operatorname{PB}^m_f(\xi)=f(w)$.

\medskip

$(1)\Rightarrow(4):$ Fix $z\in \partial \Omega$. Let $f$ be a continuous function on $\partial \Omega$ such that $f(z)=0$ and $f(\xi)<0$ for $\xi \neq z$. Then $\operatorname{PB}^m_f\in \mathcal {SH}_m(\Omega)\cap \mathcal C(\bar\Omega)$, and $\operatorname{PB}^m_f=f$ on $\partial \Omega$. Let $\mu\in \mathcal {J}_z^m$ then, since $\mu$ is a probability measure on $\bar \Omega$, we have that
\[
\operatorname{PB}^m_f(z)\leq \int\operatorname{PB}^m_f\, d\mu\leq \left (\max_{\operatorname{supp}(\mu)}\operatorname{PB}^m_f\right)\int\,d\mu=\operatorname{PB}^m_f(z)\, .
\]
Thus, $\mu=\delta_z$.

\medskip

$(4)\Rightarrow(1):$ This follows from Corollary~\ref{cor3}.

\medskip

$(1)\Rightarrow(3):$ Take $f(z)=-2|z|^2$ on $\partial \Omega$ and set $u(z)=\operatorname{PB}^m_f(z)+|z|^2$.
By Richberg's approximation theorem we can find a smooth function $v$ that is $m$-subharmonic and
\[
\lim_{\xi\to \partial \Omega}\left(u(\xi)- v(\xi)\right)=0\, .
\]
This implication is then concluded by letting $\varphi(z)=v(z)+|z|^2$. Some comments on Richberg's approximation theorem are in order.
In our case, Demailly's proof of Theorem~5.21 in~\cite{demailly_bok} is valid. Richberg's approximation theorem is valid in a much more abstract setting
(see e.g.~\cite{HarveyLawson4,plis}).

\medskip

$(3)\Rightarrow(1):$ Let $f\in \mathcal C(\partial \Omega)$, and let $\varepsilon >0$. Then there exists a smooth function  $g$ defined on a neighborhood of $\bar\Omega$ such that
\[
f\leq g\leq f+\varepsilon\, ,  \qquad \text{ on } \partial \Omega\, .
\]
By assumption there exists a constant $M>0$ such that $g+M\varphi\in \mathcal {SH}_m(\Omega)$. Then we have that
\[
g+M\varphi-\varepsilon \leq f\, ,  \qquad \text{ on } \partial \Omega\, .
\]
Hence, $g+M\varphi-\varepsilon\leq \operatorname{PB}^m_f$ in $\Omega$. This means that
\[
\varliminf_{\xi\to w}\operatorname{PB}^m_f(\xi)\geq g(w)-\varepsilon\geq f(w)-\varepsilon\qquad \text{ for all } w\in \partial \Omega\, ,
\]
and therefore we get
\[
\lim_{\xi\to w}\operatorname{PB}^m_f(\xi)=f(w)\, .
\]
Thus, $\operatorname{PB}^m_f\in \mathcal {SH}_m(\Omega)\cap\mathcal C(\bar\Omega)$, by the generalized Walsh theorem (Theorem~\ref{walsh}).

\end{proof}
\begin{remark}
In connection with Theorem~\ref{thm_mhx} and Theorem~\ref{bm-reg} we should mention~\cite{Gogus}, and~\cite{GogusSahutoglu}.
\end{remark}

An immediate consequence of Theorem~\ref{bm-reg} is the following corollary.

\begin{corollary}
Let $\Omega$ be a bounded domain in $\mathbb C^n$ such that for every $z\in \partial \Omega$ there exists a neighborhood $U_z$ such that $\Omega\cap U_z$ is $B_m$-regular, then $\Omega$ is $B_m$-regular.
\end{corollary}
\begin{proof}
Let $z\in \partial \Omega$, $U_z$ be a neighborhood of $z$, and let $u_z$ be a strong barrier at $z$, that is $m$-subharmonic, and defined in some neighborhood of $\bar U_z\cap \Omega$. Now let $\delta>0$, be such that $u_z<-\delta$ on $\partial U_z\cap \Omega$. Then we can define a (global) strong barrier at $z$, that is $m$-subharmonic:
\[
v_z(w)=\begin{cases}
\max\{u_z(w),-\delta\} & \text {if } w\in U_z\cap \Omega,\\
-\delta & \text {if }  w\in \Omega\setminus U_z\, .
\end{cases}
\]
\end{proof}

\section{The existence of smooth exhaustion functions}\label{sec_smooth}

 The purpose of this section is to prove the implication $(1)\Rightarrow(3)$ in Theorem~\ref{thm_mhx}. That we shall do in Theorem~\ref{smoothexh}. This section
 is based on the work of Cegrell~\cite{C}, and therefore shall need a few additional preliminaries.
\begin{definition}
Assume that $\Omega$ is a bounded domain in $\mathbb C^n$, and let $u\in \mathcal {SH}_m(\Omega)\cap L^{\infty}(\Omega)$. Then the \emph{$m$-Hessian measure} of $u$ is defined by
\[
\operatorname{H}_m(u)=(dd^cu)^m\wedge \beta^{n-m}\, .
\]
where $\beta=dd^c|z|^2$.
\end{definition}
\begin{remark}
The $m$-Hessian measure is well-defined for much more general functions than needed in this section. For further information see e.g.~\cite{Blocki_weak}.
\end{remark}

For a bounded $m$-hyperconvex domain in $\C^n$ we shall use the following notation
\[
\mathcal E_m^0(\Omega)=\bigg\{\varphi\in\mathcal {SH}_m(\Omega)\cap L^\infty (\Omega): \,\varphi\leq 0,\,\, \lim_{z\to\partial\Omega}\varphi (z)=0, \int_\Omega \operatorname{H}_m(\varphi)<\infty\bigg\}\, .
\]

 In Theorem~\ref{smoothexh} we shall prove that a $m$-hyperconvex domain admits an exhaustion function that is smooth, and strictly $m$-subharmonic. Our method is
  that of approximation. Therefore, we first need to prove a suitable approximation theorem. Theorem~\ref{app2} was first proved in the case $m=n$ by Cegrell~\cite{C}. If the approximating sequence $\{\psi_j\}$ only is continuous on $\Omega$, then the corresponding result was proved by
Cegrell~\cite[Theorem 2.1]{cegrell_gdm} in the case $m=n$, and Lu~\cite[Theorem 1.7.1]{L} for general $m$. In connection with Theorem~\ref{app2} we would
like to make a remark on Theorem~6.1 in a recent paper by Harvey et al.~\cite{HarveyLawsonPlis}. There they prove a similar approximation theorem, but
there is an essential difference. They assume that the underlying space should admit a negative exhaustion function that is $C^2$-smooth, and strictly $m$-subharmonic.  Thereafter, they prove that approximation is possible. Whereas we prove that smooth approximation is always
possible on an $m$-hyperconvex domain, i.e. there should only exist a negative exhaustion function. Thereafter we prove the existence of a negative and smooth
exhaustion function that is strictly $m$-subharmonic, and has bounded $m$-Hessian measure. We believe that Theorem~\ref{app2} is of interest in its own right.

\begin{theorem}\label{app2}
Assume that $\Omega$ is a bounded $m$-hyperconvex domain in $\mathbb C^n$. Then, for any negative $m$-subharmonic function $u$ defined on $\Omega$, there exists a decreasing sequence $\{\psi_j\}\subset \mathcal E_m^0(\Omega)\cap\mathcal C^{\infty}(\Omega)$
 such that $\psi_j\to u$, as $j\to \infty$.
\end{theorem}

Before proving Theorem~\ref{app2} we need the following lemma. The proof is as in~\cite{C}, and therefore it is omitted.

\begin{lemma}\label{corapp}
Let $u,v$ be smooth $m$-subharmonic functions in $\Omega$ and let $\omega$ be a neighborhood of the set $\{u=v\}$. Then there exists a smooth $m$-subharmonic function $\varphi$ such that $\varphi\geq \max \{u,v\}$
on $\Omega$ and $\varphi=\max \{u,v\}$ on $\Omega\setminus \omega$.
\end{lemma}

Now to the proof of Theorem~\ref{app2}.

\begin{proof}[Proof of Theorem~\ref{app2}]
By Theorem~\ref{thm_prop}, property (3), we can always find a continuous and negative exhaustion function $\alpha$ for $\Omega$ that is strictly $m$-subharmonic.

\medskip

We want to  prove that for any $u\in \mathcal E_m^0(\Omega)\cap \mathcal C(\bar \Omega)$ with $\operatorname {supp}(\operatorname{H}_m(u))\Subset \Omega$, and for any $a\in (1,2)$, there exists $\psi\in \mathcal E_m^0(\Omega)\cap\mathcal C^{\infty}(\Omega)$
such that
\begin{equation}\label{eq1}
au\leq \psi\leq u.
\end{equation}
We shall do it in several steps.

\emph{Step 1.} Fix a constant $s<0$ such that
\[
\operatorname {supp}(\operatorname{H}_m(u))\Subset \Omega_0=\{z \in \Omega: \alpha(z)<s\}\, ,
\]
and let $1<b<a<2$ and $c<0$ be constants such that
$au<bu+c$ in a neighborhood of $\bar \Omega_0$. Note that we have
\[
\bar{\Omega}_0\subset \{au<c\}\subset \{2u<c\}\, .
\]
By using standard regularization by convolution (Theorem \ref{thm_basicprop1}) we can construct a sequence $\phi_j'$ of smooth $m$-subharmonic functions decreasing to $bu$. Out of this sequence pick one function, $\varphi_0'$, that is smooth in a neighborhood of the set $\{2u\leq c\}$, and such that $\varphi_0'<u$ on
$\bar{\Omega}_0$. Next, define
\[
\varphi_0=\begin{cases}
\max\{2u,\varphi_0'+c\} & \text{ on }  \{2u<c\}, \\
2u, &\text{ on } \{2u\geq c\}.
\end{cases}
\]
Then by construction we have that $\varphi_0\in\mathcal{E}_m^0(\Omega)\cap\mathcal C(\bar\Omega)$. Furthermore, on a neighborhood of $\bar{\Omega}_0$ we have $\varphi_0=\varphi_0'+c$, since
\[
2u<au<bu+c<\varphi_0'+c\, .
\]
With the definition
\[
\tilde \varphi_0=\sup \{v\in \mathcal {SH}_m(\Omega): v\leq \varphi_0 \,\text { on } \, \Omega_0, v\leq 0\}\, ,
\]
we get that $\tilde \varphi_0=\textbf{S}_f$, where
\[
f=\begin{cases}
\varphi_0 & \text{ on } \, \Omega_0, \\
h & \text{ on } \, \bar \Omega\setminus \Omega_0\, ,
\end{cases}
\]
is a continuous function. Here $h$ is the unique harmonic function on $\Omega\setminus \Omega_0$ that is continuous up to the boundary, $h=\varphi_0$ on $\partial \Omega_0$ and $h=0$ on $\partial \Omega$. Thanks to the generalized Walsh theorem (Theorem~\ref{walsh}) we have
that $\tilde \varphi_0\in \mathcal {SH}_m(\Omega)\cap \mathcal C(\bar \Omega)$.  Furthermore,
\[
au<bu+c\leq \varphi_0'+c=\varphi_0=\tilde \varphi_0<\varphi_0'<u\qquad  \text{ on }\bar \Omega_0\, .
\]
Thus, we see that
\[
au<\tilde \varphi_0<u \qquad \text { on } \bar \Omega\, .
\]
The set $\{au\leq \varphi_0\}\subset \{2u\leq c\}$ is compact, and therefore we have that $\varphi_0$ is smooth in a neighborhood of $\{au\leq \varphi_0\}$.

\medskip

\emph{Step 2.} Let $\Omega_0'$ be a given domain such that $\Omega_0\Subset \Omega_0'\Subset \Omega$. We shall construct functions $\varphi_1$, $\tilde \varphi_1$, and a domain $\Omega_1$ with the following properties;
\begin{enumerate}\itemsep2mm
\item $\Omega_0'\Subset \Omega_1\Subset \Omega$ and $\Omega_1=\{\alpha<s_1\}$, for some $s_1<0$;
\item $\varphi_1,\tilde \varphi_1\in \mathcal E_m^0(\Omega)\cap\mathcal C(\bar \Omega)$;
\item $\varphi_0=\varphi_1$ on $\Omega_0$;
\item $au<\tilde \varphi_1<u$ on $\Omega$;
\item $\varphi_1=\tilde \varphi_1$ on $\Omega_1$;
\item $\{au\leq \varphi_1\}\Subset \Omega$ and
\item $\varphi_1$ is smooth in a neighborhood of $\{au\leq \varphi_1\}$.
\end{enumerate}
We start by taking $s_1<0$ such that
\[
\Omega_0'\Subset \Omega_1=\{\alpha<s_1\}\Subset \Omega\, .
\]
and $\varphi_0<au$ on $\partial \Omega_1$. This is possible since the set
$\{au\leq \varphi_0\}$ is compact. Let $1<b<a$, and $c<d<0$, with the properties that
\[
au<bu+d<\tilde \varphi_0 \, \text{ on a neighborhood of }\, \bar \Omega_1\, .
\]
Once again using standard approximation by convolutions, let $\phi_j''$ be a sequence of smooth $m$-subharmonic functions decreasing to $bu+d$. Take one function from this sequence, call it $\varphi_1''$, such that it is smooth in a neighborhood of $\{2u\leq d\}$, and
\[
\varphi_1''<\tilde \varphi_0 \qquad \text{ on } \bar \Omega_1\, .
\]
The definition
\[
\varphi_1'=\begin{cases} \max \{\varphi_1'',2u\} & \text{ on } \{2u<d\}, \\
2u & \text{ on }  \{2u\geq d\}
\end{cases}
\]
yields that $\varphi_1'\in \mathcal E_m^0(\Omega)\cap\mathcal C(\bar\Omega)$, and we have that $\varphi_1'=\varphi_1''$ near $\{au\leq \varphi_1'\}$.

Take an open set $W$ such that
\[
\{au<\varphi_0=\varphi_1'\}\Subset W\Subset \{au<\min(\varphi_0,\varphi_1')\}\setminus \bar \Omega_0\, ,
\]
therefore by Lemma~\ref{corapp} there exists $\varphi_1\in \mathcal E_m^0(\Omega)$ such that $\varphi_1<u$ on $\Omega$, and with $\varphi_1\geq \max\{\varphi_0,\varphi_1'\}$ with equality on $\Omega_0$. Furthermore,
$\varphi_1$ is smooth on $W$ and $\varphi_1=\varphi_0$ on $\Omega_0$. It also follows that $\varphi_1$ is smooth near $\{au\leq \varphi_1\}$ which contains $\bar \Omega_1$, since $\varphi_1=\varphi_1'$
if $\varphi_0<au\leq \varphi_1$. Both functions $\varphi_0$, and $\varphi_1'$,  are smooth near
\[
\{au\leq \varphi_0\}\cap \{au\leq \varphi_1'\}\, .
\]
Let us define
\[
\tilde \varphi_1=\sup \{v\in \mathcal {SH}_m(\Omega): v\leq \varphi_1 \,\text { on } \, \Omega_1, v\leq 0\}\, ,
\]
then as in \emph{Step 1} it follows that $\tilde \varphi_1\in \mathcal {SH}_m(\Omega)\cap \mathcal C(\bar \Omega)$. The constructions
$\varphi_1$, $\tilde \varphi_1$ and $\Omega_1$ satisfy  all the conditions (1)-(7).

\medskip

\emph{Step 3}. Now if $\Omega_j\nearrow \Omega$, then the function
\[
\psi=\lim_{j\to \infty}\varphi_j\in\mathcal E_m^0(\Omega)\, .
\]
Furthermore, $\psi$ is smooth since  for any domain $\omega\Subset\Omega$ there exists $j_{\omega}$ such
that on the set $\omega$ we have $\psi=\varphi_{j_{\omega}}\in \mathcal C^{\infty}$. This ends the proof of (\ref{eq1}).

To finish the proof of this theorem, assume that $u$ is a negative
$m$-subharmonic function defined on $\Omega$. Theorem 1.7.1 in~\cite{L} implies that there exists a decreasing sequence $\{u_j\}\subset \mathcal E_m^0(\Omega)\cap \mathcal C(\bar \Omega)$, $\operatorname{supp}(\operatorname{H}_m(u_j))\Subset \Omega$, such that $u_j\to u$, as $j\to \infty$. Then by~(\ref{eq1}) there exists a sequence $\psi_j\in \mathcal E_m^0(\Omega)\cap\mathcal C^{\infty}(\Omega)$ with
\[
\left(1-\frac {1}{j+1}\right)u_j\leq \psi_j\leq \left(1-\frac 1j\right) u_j\, ,
\]
and the proof is finished.
\end{proof}

We shall end this paper by proving the implication $(1)\Rightarrow(3)$ in Theorem~\ref{thm_mhx}.

\begin{theorem}\label{smoothexh}
 Assume that $\Omega$ is a $m$-hyperconvex domain in $\C^n$, $n\geq 2$, $1\leq m\leq n$.  Then $\Omega$ admits an exhaustion function that is
 negative, smooth, strictly $m$-subharmonic, and has bounded $m$-Hessian measure.
\end{theorem}
\begin{proof} Theorem~\ref{app2} implies that there exists a function $\psi\in \mathcal E_m^0(\Omega)\cap \mathcal C^{\infty}(\Omega)$. Let $M>0$ be a constant such that
\[
|z|^2-M<-1 \qquad  \text{ on } \Omega\, ,
\]
and define
\[
\psi_j(z)=\max \left\{\psi(z),\frac {|z|^2-M}{j}\right\}\in \mathcal E_m^0(\Omega)\cap\mathcal C(\Omega)\, .
\]
This construction also implies that $\psi_j$ is smooth outside a neighborhood $\omega$ of the set
\[
\left\{\psi(z)=\frac {|z|^2-M}{j}\right\}\, .
\]
Lemma~\ref{corapp} implies that there exists $\varphi_j\in \mathcal E_m^0(\Omega)\cap\mathcal C^{\infty}(\Omega)$ such that $\varphi_j=\psi_j$ outside $\omega$.
Now we choose a sequence $a_j\in (0,1)$ such that the function
\[
\varphi=\sum_{j=1}^{\infty}a_j\varphi_j
\]
is smooth, strictly $m$-subharmonic, and belongs to $\mathcal E_m^0(\Omega)$. It is sufficient to take
\[
a_j=\frac {1}{2^j\max\left\{\|\varphi_j\|_{\infty}, h_j^{\frac 1m}, 1\right\}}\, , \quad \text { where } \, h_j=\int_{\Omega}\operatorname{H}_m(\varphi_j).
\]
Note here that $|\varphi|\leq 1$. The construction
\[
u_n=\sum_{j=1}^na_j\varphi_j
\]
implies that $u_n\in \mathcal E_m^0(\Omega)$, and $u_n\searrow \varphi$, as $n\to\infty$.  Using standard arguments, and finally by passing to the limit with $n$,
we arrive at
\[
\int_{\Omega}\operatorname{H}_m(\varphi)\leq \left(\sum_{j=1}^{\infty}a_j\left(\int_{\Omega}\operatorname{H}_m(\varphi_j)\right)^{\frac 1m}\right)^m\leq 1\, .
\]
Let us conclude this proof by motivating why $\varphi$ is necessarily smooth, and strictly $m$-subharmonic. Let $\Omega'\Subset \Omega$, then there exists an index $j_\omega$ such that on $\Omega'$ we have that
\[
\varphi_j=\frac{|z|^2-M}{j}\qquad \text{ for } j>j_\omega\, .
\]
This gives us that
\[
\varphi=\sum_{j=1}^{j_\omega}a_j \varphi_j+\sum_{j=j_\omega+1}^{\infty}a_j \left(\frac{|z|^2-M}{j}\right) \qquad \text { on }\, \Omega'\, .
\]
\end{proof}


\begin{thebibliography}{99}

\bibitem{Abdullaev1} Abdullaev B. I., Subharmonic functions on complex hyperplanes of $\C^n$. J.
Sib. Fed. Univ. Math. Phys. (2013), Volume 6, Issue 4, 409-416.

\bibitem{Abdullaev2} Abdullaev B. I., $\mathcal{P}$-measure in the class of $m-wsh$ functions. J. Sib. Fed.
Univ. Math. Phys. (2014), Volume 7, Issue 1, 3-9.

\bibitem{SA}  Abdullaev B. I., Sadullaev A., Potential theory in the class of $m$-subharmonic functions. Proc. Steklov Inst. Math.  279  (2012),  no. 1, 155-180.

\bibitem{armitage} Armitage D. H., Gardiner S. J., Classical potential theory. Springer Monographs in Mathematics. Springer-Verlag London, Ltd., London, 2001.

\bibitem{avelin} Avelin B., Hed L., Persson H., A note on the hyperconvexity of pseudoconvex domains beyond Lipschitz regularity. Potential Anal.  43  (2015),  no. 3, 531-545.

\bibitem{aytuna} Aytuna A., On Stein manifolds $M$ for which $\mathcal{O}(M)$  is isomorphic to $\mathcal{O}(\Delta^n)$ as Fr\'{e}chet spaces.
Manuscripta Math. 62 (1988),  no. 3, 297-315.

\bibitem{Blocki_LN} B{\l}ocki Z., The complex Monge-Amp\`{e}re operator in pluripotential theory. Lecture Notes, 2002.
(\texttt{http://gamma.im.uj.edu.pl/$\sim$blocki/}).

\bibitem{Blocki_MA} B{\l}ocki Z., The complex Monge-Amp\`{e}re operator in hyperconvex domains. Ann. Scuola Norm. Sup. Pisa Cl. Sci. (4) 23 (1996),  no. 4, 721-747.

\bibitem{Blocki_weak} B{\l}ocki Z., Weak solutions to the complex Hessian equation. Ann. Inst. Fourier (Grenoble)  55  (2005),  no. 5, 1735-1756.

\bibitem{bouligand} Bouligand G., Sur le probl\`{e}me de Dirichlet. Ann. Soc. Polon. Math. 4 (1926), 59-112.

\bibitem{brelot} Brelot M., Familles de Perron et probl\`{e}me de
 Dirichlet. Acta. Litt. Sci. Szeged  9 (1939), 133-153.

\bibitem{brelot2} Brelot M., On topologies and boundaries in potential theory. Lecture Notes in Mathematics, Vol. 175 Springer-Verlag, Berlin-New York 1971.

\bibitem{bremermann} Bremermann H. J., On a generalized Dirichlet
  problem for plurisubharmonic functions and pseudo-convex domains.
Characterization of \u{S}ilov boundaries.
  Trans. Amer. Math. Soc. 91 (1959), 246-276.

\bibitem{catlin} Catlin D. W., Global regularity of the $\bar{\partial}$-Neumann problem. Complex analysis of several variables (Madison, Wis., 1982),  39-49, Proc. Sympos. Pure Math., 41, Amer. Math. Soc., Providence, RI, 1984.

\bibitem{carat} Carath\'{e}odory C., On the Dirichlet's Problem.
Amer. J. Math. 59 (1937), no. 4, 709-731.


\bibitem{CHS} Caffarelli L., Nirenberg L., Spruck J., The Dirichlet problem for nonlinear second-order elliptic equations. III. Functions of the eigenvalues of the Hessian. Acta Math. 155 (1985),  no. 3-4, 261-301.

\bibitem{CCW} Carlehed M., Cegrell U., Wikstr\"om F., Jensen measures, hyperconvexity and boundary behaviour of the pluricomplex Green function. Ann. Polon. Math. 71 (1999), no. 1, 87-103.

\bibitem{cegrell_gdm} Cegrell U., The general definition of the complex Monge-Amp\`{e}re operator. Ann. Inst. Fourier (Grenoble)  54 (2004),  no. 1,  159-179.

\bibitem{C} Cegrell, U., Approximation of plurisubharmonic functions in hyperconvex domains.  Complex analysis and digital geometry,  125-129, Acta Univ. Upsaliensis Skr. Uppsala Univ. C Organ. Hist., 86, Uppsala Universitet, Uppsala, 2009.

\bibitem{ColeRansford1} Cole B. J., Ransford T. J., Subharmonicity without upper semicontinuity. J. Funct. Anal.  147  (1997),  no. 2, 420-442.

\bibitem{ColeRansford2} Cole B. J., Ransford T. J., Jensen measures and harmonic measures. J. Reine Angew. Math.  541  (2001), 29-53.

\bibitem{demailly_bok} Demailly J.-P., Complex analytic and differential geometry. Self published e-book.\\
(\texttt{http://www-fourier.ujf-grenoble.fr/$\sim$demailly/}).

\bibitem{DinewDinew} Dinew S., Ko\l odziej S., A priori estimates for complex Hessian equations. Anal. PDE  7  (2014),
no. 1, 227-244.

\bibitem{doob} Doob J. L., Classical potential theory and its probabilistic counterpart.
Grundlehren der Mathematischen Wissenschaften, 262. Springer-Verlag, New York, 1984.

\bibitem{edwards} Edwards D. A., Choquet boundary theory for certain spaces of lower semicontinuous functions. Function Algebras (Proc. Internat. Sympos. on Function Algebras, Tulane Univ., 1965), 300-309,  Scott-Foresman, Chicago, Ill, 1966.

\bibitem{DiederichFornaess} Diederich K., Forn\ae ss J.-E., Pseudoconvex domains: bounded strictly plurisubharmonic exhaustion functions. Invent. Math. 39 (1977),
no. 2, 129-141.

\bibitem{Gaarding} G\aa rding L., An inequality for hyperbolic polynomials. J. Math. Mech. 8 (1959), 957-965.

\bibitem{Gamelin}  Gamelin T. W., Uniform algebras and Jensen measures. London Mathematical Society Lecture Note Series, 32. Cambridge University Press, Cambridge-New York, 1978.

\bibitem{GamelinSibony} Gamelin T. W., Sibony N., Subharmonicity for uniform algebras. J. Funct. Anal.  35  (1980), no. 1, 64-108.

\bibitem{Gogus} G\"o\u g\"u\c s N. G., Local and global $C$-regularity. Trans. Amer. Math. Soc.  360  (2008),  no. 5, 2693-2707.


\bibitem{GogusSahutoglu} G\"o\u g\"u\c s N. G., \c Sahuto\u glu S., Continuity of plurisubharmonic envelopes in $\C^2$. Internat. J. Math.  23  (2012),  no. 12, 1250124, 12 pp.

\bibitem{HarveyLawson1} Harvey F. R., Lawson H. B., Jr., Calibrated geometries. Acta Math.  148  (1982), 47-157.

\bibitem{HarveyLawson2} Harvey F. R., Lawson H. B., Jr.,  Plurisubharmonicity in a general geometric context.  Geometry and analysis. No. 1,  363-402, Adv. Lect. Math. (ALM), 17, Int. Press, Somerville, MA, 2011.

\bibitem{HarveyLawson3} Harvey F. R., Lawson H. B., Jr., An introduction to potential theory in calibrated geometry. Amer. J. Math.  131  (2009),  no. 4, 893-944.

\bibitem{HarveyLawson4} Harvey F. R., Lawson H. B., Jr., Geometric plurisubharmonicity and convexity: an introduction. Adv. Math.  230  (2012),  no. 4-6, 2428-2456.

\bibitem{HarveyLawsonPlis} Harvey F. R., Lawson H. B., Jr., Pli\'{s} S., Smooth approximation of plurisubharmonic functions on almost complex manifolds. Math. Ann.  366  (2016),  no. 3-4, 929-940.

\bibitem{HP} Hed L., Persson H., Plurisubharmonic approximation and boundary values of plurisubharmonic functions. J. Math. Anal. Appl.  413  (2014),  no. 2, 700-714.

\bibitem{hormander} H\"{o}rmander L., Notions of convexity. Progress in Mathematics, 127. Birkh\"auser Boston, Inc., Boston, MA, 1994.

\bibitem{huang} Huang Y., Xu L., Regularity of radial solutions to the complex Hessian equations. J. Partial Differ. Equ.  23  (2010),  no. 2, 147-157.

\bibitem{jarnicki_pflug} Jarnicki M., Pflug P., First Steps in Several Complex Variables: Reinhardt Domains. EMS Textbooks in Mathematics. European Mathematical Society (EMS), Z\"{u}rich, 2008.

\bibitem{KR} Kerzman N., Rosay J.-P., Fonctions plurisousharmoniques d'exhaustion born\'{e}es et domaines taut. Math. Ann. 257 (1981), no. 2, 171-184.

\bibitem{K} Klimek M., Pluripotential theory. London Mathematical Society Monographs. New Series, 6. Oxford Science Publications. The Clarendon Press, Oxford University Press, New York, 1991.

\bibitem{landkof} Landkof N. S., Foundations of modern potential theory. Die Grundlehren der mathematischen Wissenschaften, Band 180. Springer-Verlag, New York-Heidelberg, 1972.

\bibitem{Lebesgue1} Lebesgue H., Sur des cas d'impossibilit\'{e} du probl\`{e}me de Dirichlet ordinaire. C. R. Seances
Soc. Math. France 17 (1912).

\bibitem{Lebesgue2} Lebesgue H., Conditions de r\'{e}gularit\'{e}, conditions d'irr\'{e}gularit\'{e}, conditions d'impossibilit\'{e}
dans le probl\`{e}me de Dirichlet. C. R. Acad. Sci. Paris 178 (1924), 349-354.

\bibitem{lelong}  Lelong P., D\'{e}finition des fonctions plurisousharmoniques. C. R. Acad. Sci. Paris  215 (1942), 398-400.

\bibitem{Li} Li S.-Y., On the Dirichlet problems for symmetric function equations of the eigenvalues of the complex Hessian. Asian J. Math.  8 (2004),  no. 1, 87-106.

\bibitem{L} Lu H.-C., Complex Hessian equations. Doctoral thesis, University of Toulouse III Paul Sabatier, 2012.

\bibitem{L2} Lu H.-C., Solutions to degenerate complex Hessian equations. J. Math. Pures Appl. (9) 100 (2013),  no. 6, 785-805.

\bibitem{Lukes et al}  Luke\v{s} J., Mal\'{y} J., Netuka I., Spurn\'{y} J., Integral representation theory. Applications to convexity, Banach spaces and potential theory. De Gruyter Studies in Mathematics, 35. Walter de Gruyter \& Co., Berlin, 2010.

\bibitem{N} Nguyen N. C., Subsolution theorem for the complex Hessian equation. Univ. Iagiel. Acta Math. 50 (2013), 69-88.


\bibitem{oka} Oka K., Sur les fonctions analytiques de plusieurs variables. VI. Domaines pseudoconvexes. T\^{o}hoku Math. J.  49 (1942), 15-52.

\bibitem{phong} Phong D., Picard S., Zhang X., A second order estimate for general complex Hessian equations. Anal. PDE  9  (2016),  no. 7, 1693-1709.

\bibitem{plis} Pli\'{s} S., The smoothing of $m$-subharmonic functions. Manuscript (2013),  arXiv:1312.1906.

\bibitem{Poincare} Poincar\'{e} H., Sur les \'{e}quations aux d\'{e}riv\'{e}es partielles de la physique math\'{e}matique. Amer. J. Math.  12  (1890),  no. 3, 211-294.

\bibitem{perron} Perron O., Eine neue Behandlung der erten Randwertaufgabe f\"ur
$\Delta u=0$. Math. Z.  Math. Z.  18  (1923),  no. 1, 42-54.

\bibitem{Ransford}  Ransford T. J., Jensen measures.  Approximation, complex analysis, and potential theory (Montreal, QC, 2000),  221-237, NATO Sci. Ser. II Math. Phys. Chem., 37, Kluwer Acad. Publ., Dordrecht, 2001.

\bibitem{rickart} Rickart C. E., Plurisubharmonic functions and convexity properties for general function algebras. Trans. Amer. Math. Soc. 169 (1972), 1-24.

\bibitem{stehle} Stehl\'{e} J.-L., Fonctions plurisousharmoniques et convexit\'{e} holomorphe de certains fibr\'{e}s analytiques. S\'{e}minaire Pierre Lelong (Analyse), Ann\'{e}e 1973-1974,  155-179. Lecture Notes in Math., Vol. 474, Springer, Berlin, 1975.

\bibitem{S}  Sibony N., Une classe de domaines pseudoconvexes. Duke Math. J.  55  (1987),  no. 2, 299-319.

\bibitem{TW1} Trudinger N. S., Wang X.-J., Hessian measures. I. Topol. Methods Nonlinear Anal.  10  (1997),  no. 2, 225-239.

\bibitem{TW2} Trudinger N. S., Wang X.-J., Hessian measures. II. Ann. of Math. (2)  150  (1999),  no. 2, 579-604.

\bibitem{TW3} Trudinger N. S., Wang X.-J., Hessian measures. III. J. Funct. Anal.  193  (2002),  no. 1, 1-23.

\bibitem{walsh} Walsh J. B., Continuity of envelopes of plurisubharmonic
      functions, J. Math. Mech. 18 (1968), 143-148.

\bibitem{WanWang} Wan D., Wang W., Lelong-Jensen type formula, $k$-Hessian boundary measure and Lelong number for $k$-convex functions. J. Math. Pures Appl. (9) 99 (2013),  no. 6, 635-654.

\bibitem{WanWang2} Wan D., Wang W., Complex Hessian operator and Lelong number for unbounded $m$-subharmonic functions. Potential Anal.  44  (2016),  no. 1, 53-69.

\bibitem{Wang} Wang X.-J., The $k$-Hessian equation. Geometric analysis and PDEs.  Lecture Notes in Math., 1977, 177-252, Springer, Dordrecht, 2009.

\bibitem{wiener1} Wiener N., Certain notations in potential theory. J. Math. Phys. Mass. Inst. Tech. 3 (1924), 24-51.

\bibitem{wiener2} Wiener N., The Dirichlet problem. J. Math. Phys. Mass. Inst. Tech.  3 (1924), 127-146.

\bibitem{wiener3} Wiener N., Note on a paper by O. Perron. J. Math. Phys. Mass. Inst. Tech.  4 (1925), 31-32.

\bibitem{W} Wikstr\"om, F., Jensen measures and boundary values of plurisubharmonic functions. Ark. Mat. 39 (2001), no. 1, 181-200.

\bibitem{zaremba} Zaremba S., Sur le principe de Dirichlet. Acta Math.  34  (1911),  no. 1, 293-316.

\bibitem{zhou} Zhou B., The Sobolev inequality for complex Hessian equations. Math. Z.  274  (2013),  no. 1-2, 531-549.


\end{thebibliography}
\end{document}